\documentclass[12pt]{amsart}
\usepackage{a4wide}
\usepackage{url}
\usepackage{amssymb}
\usepackage{mathbbol}

\newcommand{\abs}[1]{\left\lvert #1 \right\rvert}  
\DeclareMathOperator{\supp}{{\mathrm supp}}                   

\newcommand{\CC}{\mathcal{A}}

\newcommand\Prob{\mathbf{P}}
\newcommand{\CG}{{\mathcal{C}}} 
\newcommand{\CX}{\mathcal{X}} 
\newcommand{\LL}{\mathcal{L}} 
\newcommand\bp{{p}}
\newcommand{\f}{d} 
\newcommand{\opT}{\mathfrak{T}} 
\newcommand{\clA}{F} 
\newcommand{\eCm}{e_{\CG_m}} 
\newcommand{\trees}{\mathcal{T}} 

\DeclareMathOperator{\IE}{\mathbf{E}}                     
\DeclareMathOperator{\Tr}{\mathrm{Tr}}                     

\renewcommand{\phi}{\varphi{}}

\newcommand{\IC}{\mathbf{C}}                     
\newcommand{\IF}{\mathbf{F}}                     
\newcommand{\IZ}{\mathbf{Z}}                     
\newcommand{\IQ}{\mathbf{Q}}                     
\newcommand{\IP}{\mathbf{P}}                     

\makeatletter
  \DeclareRobustCommand\em
         {\@nomath\em \ifdim \fontdimen\@ne\font >\z@
                       \upshape \else \slshape \fi}
\makeatother

\newtheoremstyle{mythm}
  {9pt}
  {9pt}
  {\slshape}
  {0pt}
  {\bfseries}
  {.}
  { }
  {\thmname{#1} \thmnumber{ #2}\thmnote{ (#3)}}

\theoremstyle{mythm}

\newtheorem{theorem}{Theorem}[section]
\newtheorem{lemma}[theorem]{Lemma}
\newtheorem{proposition}[theorem]{Proposition}

\theoremstyle{remark}
\newtheorem{remark}[theorem]{Remark}

\theoremstyle{definition}
\newtheorem{definition}[theorem]{Definition}

\numberwithin{equation}{section}

\begin{document}
\title{Free Lamplighter Groups and a Question of Atiyah}
\author{Franz Lehner}
\address{Institute for  Mathematical  Structure Theory\\
Graz Technical University\\
Steyrergasse 30, 8010 Graz, Austria}
\email{lehner@math.tugraz.at}
\author{Stephan Wagner}\thanks{This material is based upon work supported by
  the National Research Foundation of South Africa under grant number 70560.}
\address{
  Department of Mathematical Sciences\\
Stellenbosch University\\
Private Bag X1\\
Matieland 7602\\
South Africa
}
\email{swagner@sun.ac.za}
\begin{abstract}
  We compute the von Neumann dimensions of the kernels  of 
  adjacency operators 
  on free lamplighter groups and show that they are irrational,
  thus providing an elementary constructive answer
  to a question of Atiyah.
\end{abstract}


\date{\today}
\maketitle{}

\section{Introduction}

In 1976 Michael Atiyah
\cite{Atiyah:1976:elliptic}
introduced $L^2$-cohomology and $L^2$-Betti numbers
of manifolds with non-trivial fundamental group
and asked whether these numbers can be irrational.
For background on these concepts we refer to
the book~\cite{Lueck:2002:invariants}.

It was shown later that 
Atiyah's question is equivalent to a purely analytic problem
on group von Neumann algebras
and we will exclusively work in this context.
Therefore no knowledge of cohomology theory and geometry is required
in the present paper, and in the remainder of this section
we recall a few basic facts from von Neumann algebra theory
which are necessary to state Atiyah's question.
Let $\Gamma$ be a finitely presented discrete group and $\IQ\Gamma$ its rational group
ring.
An element of this ring can be represented as a bounded left convolution
operator on the space $\ell_2(\Gamma)$ of square summable functions
on $\Gamma$. Denote this representation by $\lambda$.
The equivalent formulation of Atiyah's question in this
setting is:

\emph{
Let $\opT\in M_n(\IQ\Gamma)$ be a symmetric element. Is it possible that 
the von Neumann dimension $\dim_{L(\Gamma)} \ker \lambda(\opT)$ is an irrational
number?
}

Here the \emph{von Neumann dimension} is to be understood in the setting
of the group von Neumann algebra, denoted $L(\Gamma)$,
which is the completion of the rational convolution operator algebra in the
weak operator topology. More precisely, one looks for matrices
of convolution operators, but for simplicity, in the present paper we will
work in $L(\Gamma)$ exclusively.
A symmetric element $\opT\in\IQ\Gamma$ gives rise to a selfadjoint
convolution operator $\lambda(\opT)$, all of whose spectral projections
lie in the von Neumann algebra $L(\Gamma)$.
On $L(\Gamma)$ there is a normal
faithful trace $\tau(\opT) = \langle \opT\delta_e,\delta_e\rangle$
and the value of $\tau(p)$ at projections $p\in L(\Gamma)$
is called the \emph{von Neumann dimension function}.
The von Neumann dimension of a closed subspace of $\ell_2(\Gamma)$
is the von Neumann dimension of the corresponding orthogonal
projection, if the latter happens to be an element of $L(\Gamma)$.

While rationality of kernel dimensions has been shown for many
examples, (see, e.g., \cite {Linnell:1993:division,LinnellSchick:2007:finite}),
recently Atiyah's question was answered affirmatively
by Tim Austin~\cite{Austin:2009:rational}
who constructed an uncountable family of groups and rational
convolution operators with distinct kernel dimensions,
which a fortiori must contain irrational numbers, even transcendental ones.
Subsequently more constructive answers were given in
\cite{Grabowski:2010:turing,PichotSchickZuk:2010:closed}.
A stronger variant of Atiyah's question had been solved earlier
\cite{GrigorchukLinnellSchickZuk:2000:Atiyah,
GrigorchukZuk:2001:lamplighter,DicksSchick:2002:spectral}. 
In  these examples, so-called lamplighter groups play a central role.
Although these are not finitely presented,
they are recursively presentable (see, e.g., \cite{Baumslag:2005:embedding})
and therefore by a theorem of Higman \cite{Higman:1961:subgroups}
can be embedded into finitely presented groups.

In the present paper, complementing the above results,
we pursue the ideas
of~\cite{DicksSchick:2002:spectral}
and
compute explicitly the von Neumann dimension of the kernel of the
``switch-walk-switch'' adjacency 
operators on the free lamplighter groups $\CG_m\wr\IF_\f$ 
with respect to the canonical generators
and show that
they are irrational for any $\f\geq 2$ and $m>2\f-1$. 
This provides another elementary explicit example of a rational convolution
operator with irrational kernel dimension.

The basic ingredient is Theorem~\ref{thm:lamplighterpercolation},
which generalises the methods of Dicks and
Schick~\cite{DicksSchick:2002:spectral} from the infinite cyclic group
to arbitrary discrete groups and makes a link to percolation theory,
thus providing a quite explicit description of the spectrum of
switch-walk-switch
transition operators on lamplighter groups as the union of the spectra of
all finite connected subgraphs of the Cayley graph.
In particular, the lamplighter kernel dimension equals the expected
normalised
kernel dimension of the percolation cluster.

The paper is organised as follows.
In Section~\ref{sec:lamplighter} we review
the necessary prerequisites about lamplighter groups and percolation
and state the main result.

In Section~\ref{sec:matchings} we recall the connection between
spectra and matchings of finite trees and compute the generating
function of the kernel dimensions of finite subtrees of
the Cayley graph of the free group. As a final step
we integrate this generating function
in Section~\ref{sec:parametrisation} and thus obtain the
dimensions we are interested in. Another example of a free
product of groups whose Cayley graph is not a tree is discussed in
Section~\ref{sec:freeproduct}.

\emph{Acknowledgements.}
We thank Slava Grigorchuk for explaining Atiyah's question,
Martin Widmer and Christiaan van de Woestijne  for discussions about transcendental numbers,
and Mark van Hoeij for a hint to compute a certain abelian integral,
which ultimately led to the discovery of
the parametrisation \eqref{eq:parametrisation}.
Last but not least  we thank three anonymous referees
for numerous remarks which helped to improve the presentation.

\section{Lamplighter groups and percolation}
\label{sec:lamplighter}
Let $G$ be a discrete group and fix a symmetic generating set $S$.
We denote by $\CX=\CX(G,S)$ the Cayley graph
of $G$ with respect to $S$ and in the rest of the paper
identify the rational group algebra
element $\opT=\sum_{s\in S}s$ with the corresponding convolution operator,
which coincides with the adjacency operator on $\CX$.

\subsection{Lamplighter groups}
The name \emph{lamplighter group} has been coined in  recent years
to denote   wreath products of the form $\Gamma=\CG_m\wr G$.
This is the semidirect product
$\LL\rtimes G$, where $\CG_m$ is the cyclic group of order $m$
and $\LL=\bigoplus_G \CG_m$ is the group
of \emph{configurations} $\eta:G\to \CG_m$ with finite support,
where we define $\supp \eta = \{x\in G : \eta(x) \ne \eCm\}$.
The group operation on $\LL$ is pointwise multiplication in $\CG_m$
and the natural
left action of $G$ on $\LL$ given by $L_g\eta(x) = \eta(g^{-1}x)$
induces the twisted group law on $\CG_m\wr G$ 
$$
(\eta,g)(\eta',g') = (\eta\cdot L_g\eta',gg')
$$
Certain random walks on $\CG_m\wr G$ can be interpreted as
a lamplighter walking around on $G$ and turning
on and off lamps. A pair $(\eta,g)$ encodes both
the position of the lamplighter as an element $g\in G$ 
and the states of the lamps as a function $\eta\in\LL$.

We will consider here the ``switch-walk-switch'' lamplighter adjacency operator
$$
\tilde{\opT} = \sum_{s\in S} EsE
$$
on the lamplighter group $\Gamma$ where $E=\frac{1}{m}\sum_{h\in\CG_m} h$
is the idempotent corresponding to the uniform distribution on the lamp group
$\CG_m$. 
The \emph{underlying} convolution operator on $G$ is $\opT=\sum_{s\in S} s$.
Here we identify $\CG_m$ and $G$ with subgroups of $\Gamma$ via
the respective embeddings
$$
\begin{aligned}
\CG_m&\to \Gamma & \qquad\qquad G &\to\Gamma \\
 h &\mapsto (\delta_e^h,e)   & g &\mapsto (\iota, g)
\end{aligned}
$$
where $\iota$ is the neutral element of $\LL$ and 
$$
\delta_g^h(x) =
\begin{cases}
  h & x=g\\
  \eCm & x\ne g
\end{cases}
.
$$
\subsection{Percolation clusters}

Let $\CX=(V,E)$ be a graph. We use the standard notation
``$x\in\CX$'' for vertices and $x\sim y$ for the neighbour relation.
Fix a parameter $0 < \bp < 1$. In \emph{Bernoulli
site percolation} with parameter $\bp$ on $\CX$, we have i.i.d.~Bernoulli random
variables $Y_x\,$, $x \in \CX\,$, sitting at the vertices of $\CX$, with
$$
\Prob_{\bp}[Y_x=1] = \bp ,\qquad \Prob_{\bp}[Y_x=0] = q := 1-\bp\,.
$$
We can realise those random variables on the probability space 
$\Omega=\{0,1\}^\CX$ with a suitable probability measure $\Prob$.
Given $\omega\in\Omega$, denote by $\CX(\omega)$  the full subgraph
of $\CX$ induced on $\{x : Y_x(\omega)=1\}$ and for any vertex $x\in \CX$, denote by
$C_x(\omega)$ the connected component of $\CX(\omega)$ containing the vertex $x$,
which is called the \emph{percolation cluster} at $x$.
It is well known that for every connected graph
there is a critical parameter $\bp_c$
such that for any vertex $x$ a phase transition occurs in the sense that
for $\bp<\bp_c$ the cluster $C_x$ is almost surely finite and for
$\bp>\bp_c$ it is infinite with positive probability.
In order to make use of this fact we recall a combinatorial interpretation
of criticality.
\begin{definition}
  For a subset $\clA\subseteq \CX$ we denote its \emph{vertex boundary}
  $$
  d\clA = \{y\in \CX : y\not\in \clA, y\sim x \text{ for some $x\in \clA$}\}
  .
  $$
  For $x\in \CX$, we denote
  $$
  \CC_x=\{\clA\subseteq \CX: x\in \clA, \text{ $\clA$ finite and connected}\}
  \cup\{\emptyset\}
  $$
  the set of finite, possibly empty, path-connected neighbourhoods of $x$.
  These sets are sometimes called \emph{lattice animals}.
  The boundary of the empty animal 
  is defined to be 
  the set $\{x\}$. We denote by $\CC_x^*$ the set of animals at $x$ without
  the empty animal.
\end{definition}
The probability of a fixed $\clA\in\CC_x$ to occur as  percolation
cluster at $x$ is
$$
\IP[C_x=\clA] = \bp^{\abs{\clA}} q^{\abs{d\clA}}
;
$$
thus for $\bp<\bp_c$ we have
\begin{equation}
\label{equ:sump1-p=1}
\sum_{\clA\in \CC_x} \bp^{\abs{\clA}} q^{\abs{d\clA}} = 1
\end{equation}
because some $\clA\in \CC_x$ occurs almost surely.

Now for a fixed animal $\clA$ consider the truncated operator 
$$
\opT_\clA = P_\clA\opT P_\clA
$$
where $P_\clA$ is the orthogonal projection onto the finite dimensional
subspace $\{f\in \ell_2(\CX) : \supp f\subseteq \clA\}$.
We denote the random percolation  adjacency operator by
$$
\opT_\omega = \opT_{C_e(\omega)},
$$

and by $\dim\ker \opT_\clA$ the dimension of the kernel of $\opT_F$
as a finite matrix, while $\frac{\dim\ker\opT_\clA}{\abs{\clA}}$ will
be  the
von Neumann dimension of the kernel of $\opT_\clA$ regarded as an element
of the finite von Neumann algebra $M_{\abs{\clA}}(\IC)$ with
von Neumann trace $\frac{1}{\abs{\clA}}\Tr$.
Special care is needed for the  empty animal, for which we define both
  the cardinality of the boundary 
  and the von Neumann kernel dimension to be $1$.

Then we have the following relation between the spectrum of the lamplighter operator
$\tilde{\opT}$ and the spectra of $\opT_\clA$.

\begin{theorem}[{\cite{LehnerNeuhauserWoess:spectrum,Lehner:2009:eigenspaces}}]
  \label{thm:lamplighterpercolation}
  The spectral measure of the lamplighter adjacency operator $\tilde{\opT}$ of
  order $m$ on a Cayley graph
  $\CX$ is equal to the expected spectral measure of the random truncated
  adjacency operator $\opT_\omega$ on the percolation clusters of $\CX$ with
  percolation parameter $p=1/m$.
  In addition, if $p<p_c$, then there is a one-to-one correspondence
  between the eigenspaces of $\tilde{\opT}$ and the collection of eigenspaces
  of $\opT_\omega$ and we have the formula
  \begin{equation}
    \label{eq:dimensionformula}
  \dim_{L(\Gamma)} \ker \tilde{\opT} = \IE \frac{\dim \ker \opT_\omega}{\abs{C_x(\omega)}}
 = \sum_{\clA\in\CC_e} \frac{\dim \ker \opT_\clA}{\abs{\clA}}
 p^{\abs{\clA}}q^{\abs{d\clA}}
 .
  \end{equation}
\end{theorem}
In general it is hard to evaluate formula~\eqref{eq:dimensionformula}
because one has to compute the kernels of the adjacency matrices of
all finite
clusters. Due to the recursive structure of the Cayley tree however
it is possible to compute an algebraic equation for the generating function
$$
\sum_{T\in \CC_e^*} (\dim\ker \opT_\clA) x^{\abs{\clA}}
$$
on free groups and
to evaluate~\eqref{eq:dimensionformula} by integrating this function,
thus obtaining our main result, which concludes this section.
\begin{theorem}
  Denote $g_1,g_2,\dots,g_\f$ the canonical generators
  of the free group  $\IF_\f$ and
  consider the adjacency operator $\opT=\sum g_i+g_i^{-1}$ 
  on its Cayley graph. Then the von Neumann dimension of the kernel of
  the corresponding lamplighter operator
  $\tilde{\opT}=E\opT E$ on $\CG_m\wr\IF_\f$
  is the number
  $$
  \dim_{L(\Gamma)}\ker \tilde{\opT}
  = 1-2p+\frac{(\tau(p)-1)(2-2\f+2\f{}\tau(p))}{\tau(p)^2}
  $$
  where $p=1/m$ and $\tau(p)$ is the unique positive solution of the equation
  $t^{2\f-1}-t^{2\f-2}=p$. For $\f>1$, this is an irrational algebraic number,
  e.g., for $\f=2$ and $m=4$, the dimension is
  $$ -\frac56 - \frac{400}{3(766+258\sqrt{129})^{1/3}} +
  \frac{2(766+258\sqrt{129})^{1/3}}{3} \approx 0.850971. 
  $$
\end{theorem}
\begin{remark}
  Similar computations are possible in more general free product groups
  $G_1*G_2*\dots*G_n$, where each factor $G_i$ is a finite group
  whose Cayley graph possesses only cycles of length $\equiv 2\mod 6$,
  like the cyclic groups $\IZ_2$,  $\IZ_6$,  $\IZ_{10}$, etc.
  An example is briefly discussed in Section~\ref{sec:freeproduct}.
  It should be noted however that our technique does not work
  for nonzero eigenvalues, because in this case 
  it is more complicated to obtain the multiplicity of the eigenvalue.
  Moreover, 
  in contrast to other approaches
  (\cite{DicksSchick:2002:spectral,GrigorchukLinnellSchickZuk:2000:Atiyah}),
  it is restricted to adjacency operators,
  i.e., all group elements get the same weight.
\end{remark}

\section{Matchings, rooted trees, and generating functions}
\label{sec:matchings}
In this section we  prepare the evaluation of the series
\eqref{eq:dimensionformula} by computing a generating function.
To this end let us recall some notations.

Let $G=(V,E)$ be a finite graph. 
By \emph{characteristic polynomial} $\chi(G,x)$ 
(resp., \emph{spectrum, kernel dimension}) 
\emph{of a graph} we mean the characteristic polynomial (resp., spectrum,
kernel dimension) of its adjacency matrix.
A \emph{matching} of a finite graph is a set of disjoint edges,
i.e., every vertex occurs as an end point of at most one edge.
A \emph{perfect matching} is a matching which covers all the vertices
of the graph.
The \emph{matching polynomial} of a graph on $n$ vertices is the polynomial
$$
\sum_{j \geq 0} (-1)^j m(G,j)\, x^{n-2j},
$$
where $m(G,j)$ is the number of matchings of cardinality $j$.

It is well known (see, e.g., \cite{Godsil:1984:spectra,Cvetkovic:1988:recent})
that the characteristic polynomial of a finite tree
coincides with its matching polynomial.
Since the kernel dimension equals the multiplicity of eigenvalue zero,
which in turn is the degree of the polynomial $x^n\chi(G,x^{-1})$,
it follows immediately that the dimension of the kernel of a tree is given by
\begin{equation}\label{eq:kerdim}
\nu(T) = \dim \ker T = n - 2\mu(T),
\end{equation}
where $\mu(T)$ denotes the size of a matching of maximal cardinality in $T$. 
In particular, $\dim \ker T = 0$ if and only if $T$ has a perfect matching.

As a first step to evaluate \eqref{eq:dimensionformula}
we have to determine the generating function
$$G(x) = \sum_{T\in \CC_x^*} (\dim \ker T)\, x^{|T|} 
   = \sum_{T\in \CC_x^*} (|T| - 2 \mu(T))\, x^{|T|},$$
where the sum is taken over all nonempty animals $T$, i.e., 
connected subgraphs of the
Cayley graph of the free group $\IF_\f{}$ that contain the unit element $e$. 
To this end, we regard animals as rooted trees, with the root at $e$. 

\begin{definition}
  A \emph{ $k$-ary tree} is a planar rooted tree such that every vertex
  has at most $k$ children. Hence every vertex has degree at most $k+1$, and the root has degree at most $k$.
  A \emph{branch} of a $k$-ary tree is a rooted tree obtained
  by splitting off a neighbor of the root together with its offspring.
  Thus a $k$-ary tree can be defined recursively as
  a rooted tree with an ordered collection of $k$ possibly empty 
 branches.
\end{definition}
Thus our animals are $k$-ary trees with $k=2\f-1$, with the single exception that the root vertex may have degree $k+1$ (but all branches are $k$-ary trees according to the above definition).
For reasons which will become apparent soon
we split the family of rooted trees into two groups, 
following ideas similar to those employed in \cite{Wagner:2007:number}: 
\begin{definition}
  We say that a rooted tree is of \emph{type} $A$ if it has a maximum matching
  that leaves the root uncovered. Otherwise $T$ is of type $B$.
\end{definition}

Suppose that $T$ is of type $A$. Then it has a maximum matching that does not
cover the root and is therefore a union of maximum matchings in the various
branches of $T$. Hence if $S_1,S_2,\cdots,S_k$ are the branches of $T$, we have 
$$\mu(T) = \mu(S_1) + \mu(S_2) + \cdots + \mu(S_k).$$
We claim that in this case all $S_j$ are of type $B$.
For, suppose on the contrary that one of the branches, say $S_j$, is of type
$A$. Then we can choose a maximum matching in $S_j$ that does not cover the
root. Choose maximum matchings in all the other branches as well, and add the
edge between the roots of $T$ and $S_j$ to obtain a matching of cardinality
$$\mu(S_1) + \mu(S_2) + \cdots + \mu(S_k) + 1,$$
contradiction.
Conversely, if all branches of $T$ are of type $B$, 
then $T$ is of type $A$: clearly, the
maximum cardinality of a matching that does not cover the root is $\mu(S_1) +
\mu(S_2) + \cdots + \mu(S_k)$, so it remains to show that there are no
matchings of greater cardinality that cover the root. Suppose that such a
matching contains the edge between the roots of $T$ and $S_j$. Then, since
$S_j$ is of type $B$, the remaining matching, restricted to $S_j$, can only
contain at most $\mu(S_j) - 1$ edges. Each of the other branches $S_i$ can only
contribute $\mu(S_i)$ edges, so that we obtain a total of
$$\mu(S_1) + \mu(S_2) + \cdots + \mu(S_k)$$
edges, as claimed. This proves the following fact:
\begin{lemma}\label{lem:rec}
  \begin{enumerate}
   \item Let $T$ be a rooted tree and $S_1,\ldots,S_k$ its branches,
    then we have
$$\mu(T) = \sum_{i=1}^k \mu(S_i) + \begin{cases} 0 & \text{ if $T$ is of type $A$,} \\ 1 & \text{ otherwise.} \end{cases}$$
\item 
A rooted tree $T$ is of type $A$ if and only if all its branches are of type $B$. 
  \end{enumerate}
\end{lemma}
The only part that was not explicitly proven above is the formula for $\mu(T)$
in the case that $T$ is of type $B$. This, however, is easy as well: Clearly
the cardinality of a matching is at most $\mu(S_1) + \cdots + \mu(S_k) + 1$
(the summand $1$ accounting for the edge that covers the root). On the other
hand, $\mu(T)$ must be strictly greater than $\mu(S_1) + \cdots + \mu(S_k)$,
since there are matchings of this cardinality that do not cover the root.

\begin{remark}
  Consistently with Lemma~\ref{lem:rec}
  we define the tree $T_1$ that only consists of a single vertex to be of
  type $A$ with $\mu(T_1) = 0$ and 
  the empty tree $T_0$  to be of type $B$ with $\mu(T_0)= 0$.
  This is important for the generating functions constructed below.
\end{remark}

Since we are interested in the parameter $\nu(T) = \dim \ker T = |T| - 2\mu(T)$
rather than $\mu(T)$ itself, we first translate the above formula to a
recursion for $\nu(T)$: since $|T| = |S_1| + \cdots + |S_k| + 1$, we have
$$\nu(T) = \sum_{i=1}^k \nu(S_i) + \begin{cases} 1 & \text{ if $T$ is of type
    $A$,} \\ -1 & \text{ otherwise.} \end{cases}$$

Now let $\trees_k$, $\trees_{k,A}$, $\trees_{k,B}$ denote the set of all $k$-ary trees, $k$-ary trees of type $A$ and $k$-ary trees of type $B$ respectively.
We define the bivariate generating functions
$$
A:=A(u,x) = \sum_{T \in \trees_{k,A}} u^{\nu(T)} x^{|T|}
\qquad \text{and}
\qquad B:= B(u,x) = \sum_{T \in \trees_{k,B}} u^{\nu(T)} x^{|T|},
$$ 
the
summation being over $k$-ary trees in both cases (including the empty tree in
the case of $B$, and the one-vertex tree in the case of $A$).
Since any tree $T$ of type $A$ is a grafting of $k$ (possibly empty) branches $S_1,\ldots,S_k$ of type $B$ (which we write as $T = \bigvee_{i=1}^k S_i$),
we obtain
\begin{align*}
  A(u,x)
  &= \sum_{T\in \trees_{k,A}} u^{\nu(T)}\,x^{\abs{T}}\\
  &= \sum_{S_1,\dots,S_k\in \trees_{k,B}} u^{\nu(\bigvee_{i=1}^k
    S_i)}\,x^{\abs{\bigvee_{i=1}^k S_i}}\\
  &= \sum_{S_1,\dots,S_k\in \trees_{k,B}} u^{1+\sum_{i=1}^k\nu(S_i)}\,x^{1+\sum_{i=1}^k\abs{S_i}}\\
  &= ux \prod_{i=1}^k \sum_{S_i\in \trees_{k,B}} u^{\nu(S_i)}\,x^{\abs{S_i}}\\
  &= ux  B(u,x)^k
  .
\end{align*}

Similarly, in the case of type $B$, we get the following equation
\begin{align*}
  B(u,x)
  &= \sum_{T\in \trees_{k,B}} u^{\nu(T)}\,x^{\abs{T}}\\
  &= 1 + \sideset{}{^\prime}\sum_{S_1,\dots,S_k\in \trees_{k}} u^{\nu(\bigvee_{i=1}^k
    S_i)}\,x^{\abs{\bigvee_{i=1}^k S_i}}\\
  &= 1 + \sideset{}{^\prime}\sum_{S_1,\dots,S_k\in \trees_{k}} u^{-1+\sum_{i=1}^k\nu(S_i)}\,x^{\abs{\bigvee_{i=1}^k S_i}},\\
\intertext{where we took special care of the empty tree 
  and the remaining sum indicated by $\sum'$ runs over all $k$-tuples of trees such that
  at least one of them is not of type $B$; 
 this means that we have to subtract the sum over $k$-tuples of type B trees
 from the sum over $k$-tuples of arbitrary trees:}
B(u,x)  &= 1 + \sum_{S_1,\dots,S_k\in \trees_{k}} 
      u^{-1+\sum_{i=1}^k\nu(S_i)}
      \,
      x^{1+\sum_{i=1}^k\abs{S_i}}
    -
    \sum_{S_1,\dots,S_k\in \trees_{k,B}} 
      u^{-1+\sum_{i=1}^k\nu(S_i)}
      \,
      x^{1+\sum_{i=1}^k\abs{S_i}}
    \\
  &= 1 + \frac{x}{u}
     \biggl(
       \prod_{i=1}^k \sum_{S_i\in \trees_{k}}  u^{\nu(S_i)}\,x^{\abs{S_i}} 
       -
       \prod_{i=1}^k \sum_{S_i\in \trees_{k,B}}  u^{\nu(S_i)}\,x^{\abs{S_i}} 
     \biggr)\\
  &= 1 + \frac{x}{u}
     \biggl(
       \bigl(A(u,x)+B(u,x)\bigr)^k - B(u,x)^k
     \biggr)
     .
\end{align*}

In conclusion, we have translated the recursive description into
the following two functional equations for $A(u,x)$ and $B(u,x)$:
\begin{equation}
  \label{eq:ABequations}
  \begin{aligned}
A(u,x) &= ux B(u,x)^k, \\
B(u,x) &= 1 + \frac{x}{u} ((A(u,x)+B(u,x))^k - B(u,x)^k). \\
  \end{aligned}
\end{equation}
Finally, we obtain the following generating function for animals (the only
difference lying in the possibility that the root is allowed to have degree
$k+1 = 2\f{}$ as well and the empty tree is excluded this time):
$$F(u,x) = \sum_{\substack{ T \in\CC_e^*}} u^{\nu(T)} x^{|T|} = ux B(u,x)^{k+1} + \frac{x}{u} ((A(u,x)+B(u,x))^{k+1} - (B(u,x))^{k+1}).$$
We are mainly interested in the derivative with respect to $u$, since
$$G(x) = \frac{\partial}{\partial u} F(u,x) \Big|_{u=1} = \sum_{T \in \CC_e^*} \nu(T) x^{|T|}.$$
To save space, we will use the customary abbreviation $F_u$ etc.\ to denote
partial derivatives with respect to $u$. First note that
$$
A(u,x)+u^2B(u,x) = u^2+ux(A(u,x)+B(u,x))^k,
$$
and we can rewrite the identities~\eqref{eq:ABequations} as
\begin{equation}
  \label{eq:BkABkequations}
  \begin{aligned}
    B(u,x)^k &= \frac{A(u,x)}{ux}, \\
    (A(u,x)+B(u,x))^k &= \frac{A(u,x)+u^2(B(u,x)-1)}{ux}.
  \end{aligned}
\end{equation}
Taking the derivative of the second identity at $u=1$ we obtain
$$
A_u(1,x) + B_u(1,x)  = \frac{2(1-B(1,x))+x (A(1,x)+B(1,x))^k}{1-kx(A(1,x)+B(1,x))^{k-1}}.
$$
Using the identities~\eqref{eq:BkABkequations} we can express $F$ as
$$
F(u,x) = A(u,x)B(u,x)(1-\frac{1}{u^2}) + (A(u,x)+B(u,x))(\frac{A(u,x)}{u^2} + B(u,x)-1)
$$
and the derivative at $u=1$ is
\begin{align*}
F_u(1,x)
  &= 2A(1,x)B(1,x) + (A_u(1,x)+B_u(1,x))(A(1,x)+B(1,x)-1) \\
  &\phantom{=}+ (A(1,x)+B(1,x))(-2A(1,x)     +A_u(1,x)+B_u(1,x)) \\
  &= -2A(1,x)^2+(A_u(1,x)+B_u(1,x))(2A(1,x)+2B(1,x)-1)\\
  &= -2A(1,x)^2
     + \frac{2(1-B(1,x))+x(A(1,x)+B(1,x))^k}{1-kx(A(1,x)+B(1,x))^{k-1}}
      (2(A(1,x)+B(1,x))-1)\\
  &= -2A(1,x)^2
     + \frac{(2(1-B(1,x))+A(1,x)+B(1,x)-1)(2(A(1,x)+B(1,x))-1)}{1-k(A(1,x)+B(1,x)-1)/(A(1,x)+B(1,x))},
\end{align*}
making use of~\eqref{eq:BkABkequations} in the last step once again. So we finally obtain
\begin{multline}\label{eq:G_in_terms_of_B}
G(x) = F_u(1,x) \\
= -2A(1,x)^2 + \frac{(A(1,x)+B(1,x))(A(1,x)-B(1,x)+1)(2A(1,x)+2B(1,x)-1)}{k-(k-1)(A(1,x)+B(1,x))}.$$
\end{multline}

\section{Parametrisation}
\label{sec:parametrisation}

Recall that we are considering percolation on a $(k+1)$-regular tree, where $p = \frac{1}{m} < \frac{1}{k}$ is the percolation probability, and $q = 1-p$. For an animal $T$ (i.e., a potential percolation cluster), the size of the boundary is $|dT| = 2 + (k-1)|T|$, as can be seen immediately by induction on $|T|$. In view of the identity~\eqref{eq:dimensionformula}, we are interested in the expression
\begin{equation}\label{eq:our_constant}
\begin{split}
C(p) &= q + \sum_{T\in\CC_e^*} \frac{\dim \ker T}{|T|} p^{|T|}q^{2+(k-1)|T|} =
q + q^2 \sum_{T\in\CC_e^*} \frac{\dim \ker T}{|T|} (pq^{k-1})^{|T|} \\
&= q + q^2 \int_0^{pq^{k-1}} \frac{G(x)}{x}\,dx,
\end{split}
\end{equation}
since it gives the von Neumann dimension of the kernel of the lamplighter operator $\tilde{\opT}$ on $\CG_m\wr\IF_\f$. The summand $q$ takes care of the ``empty'' animal, i.e., the possibility that the vertex $x$ is not actually in $\CX(\omega)$ (which happens with probability $q$).

In order to compute this integral, we determine a parametrisation of $G$; since
$G$ is a rational function of $x$, $A$ and $B$, we first find such a parametrisation
for the functions $A$ and $B$. 
This is possible because the implicit equation~\eqref{eq:ABequations}
for $B$ defines an algebraic curve of genus zero.
Recall that $A = A(1,x)$ and $B = B(1,x)$ satisfy the equations
\begin{align*}
A &= xB^k, \\
B &= 1+x((A+B)^k-B^k).
\end{align*}
It turns out that the following parametrisation satisfies these two equations:
\begin{equation}
  \label{eq:parametrisation}
  \begin{aligned}
x &= (t-1)t^{k-1}(1+t^{k-1}-t^k)^{k-1}, \\
A &= \frac{t-1}{t(1+t^{k-1}-t^k)}, \\
B &= \frac{1}{t(1+t^{k-1}-t^k)}.
  \end{aligned}
\end{equation}
This parametrisation was essentially obtained by ``guessing'', i.e., finding
the parametrisation in special cases, which was done with an algorithm
by M.~van~Hoeij~\cite {vanHoeij:1994:algorithm} in the
\verb|algcurves| package of the computer algebra system
\verb|Maple|${}^\mathrm{TM}$~\cite{Maple10},
and extrapolating to the general case. Once the parametrisation has
been found, however, it is easy to verify it directly.

The two equations determine the coefficients of the expansions of $A$ and $B$ at $x = 0$ uniquely, hence they define unique functions $A$ and $B$ that are analytic at $0$. The above parametrisation provides such an analytic solution in which $t = 1$ corresponds to $x = 0$. Furthermore, the interval $[1,t_0]$, where $t_0$ is the solution of $t^k-t^{k-1} = \frac{1}{k}$, maps to the interval $[0,x_0]$ with 
$$x_0 = \frac1k \left(1 - \frac1k \right)^{k-1},$$
and the parametrisation is monotone on this interval. At $t = t_0$, it has a singularity (of square root type), which corresponds to the fact that $p = \frac1k$ is the critical percolation parameter and that $x_0$ is the radius of convergence and the smallest singularity of $A$ and $B$ (and thus in turn $G$). Therefore, the computation of~\eqref{eq:our_constant} amounts to integrating a rational function between $0$ and the unique solution $\tau(p)$ of
$$(t-1)t^{k-1}(1+t^{k-1}-t^k)^{k-1} = pq^{k-1}$$
inside the interval $[1,t_0]$. To show existence and uniqueness of $\tau(p)$, note again that the function $x(1-x)^{k-1}$ is strictly increasing on $[0,\frac{1}{k}]$ and thus maps this interval bijectively to $[0,x_0]$. Moreover, it follows that $\tau(p)$ is the unique positive solution of
$$\tau(p)^k-\tau(p)^{k-1} = p.$$
Plugging the parametrisations of $A$ and $B$ into~\eqref{eq:G_in_terms_of_B} yields
$$G(x) =\frac{(t-1)(t^{2k}(1-t)+t^{k-1}((2k+1)t^2-(4k+2)t+2k)+2)}{t^2(1+t^{k-1}-t^k)^2(1+kt^{k-1}-kt^k)}.$$
Together with
$$\frac{dx}{x} = \frac{(kt-k+1)(1+kt^{k-1}-kt^k)}{t(t-1)(1+t^{k-1}-t^k)}\,dt,$$
we finally end up with an integral which has a surprisingly simple antiderivative for arbitrary $k$. This antiderivative was also found by means of computer algebra, but can of course be checked directly to be an antiderivative:
\begin{align*}
C(p) &= q + q^2 \int_1^{\tau(p)} \frac{(kt-k+1)(t^{2k}(1-t)+t^{k-1}((2k+1)t^2-(4k+2)t+2k)+2)}{t^3(1+t^{k-1}-t^k)^3}\,dt \\
&= q + q^2 \frac{(t-1)(1-k+(k+1)t-t^{k+1})}{t^2(1+t^{k-1}-t^k)^2} \Big|_{t = \tau(p)} \\
&= q - p + \frac{(\tau(p)-1)(1-k+(k+1)\tau(p))}{\tau(p)^2},
\end{align*}
which shows that the constant $C(p)$ is always algebraic, since $p = \frac{1}{m}$ is rational in our context and $\tau(p)$ is a solution to an algebraic equation. In particular, for $k = 1$, one has $\tau(p) = 1 + p$, which yields
$$C(p) = 3 - 2p - \frac{2}{1+p}.$$
In general, however, $C(p)$ is not rational: take, for instance, $k = 3$ and $p = \frac14$, to obtain
$$C(p) = -\frac56 - \frac{400}{3(766+258\sqrt{129})^{1/3}} + \frac{2(766+258\sqrt{129})^{1/3}}{3} \approx 0.850971.$$
One can even easily prove the following:
\begin{proposition}
If $p = \frac1m$ for $m > k \geq 3$, then $C(p)$ is an irrational algebraic number.
\end{proposition}
\begin{proof}
Suppose that $C(p)$ is rational. Then 
$$\frac{(\tau(p)-1)(1-k+(k+1)\tau(p))}{\tau(p)^2} = \frac{a}{b}$$
for some coprime integers $a,b$ with $b > 0$. Hence $\tau = \tau(p)$ is a root of the polynomial
$$P(t) = (b(k+1)-a)t^2-2bkt+b(k-1).$$
Divide by $g = \gcd(b(k+1)-a,-2bk,b(k-1))$ to obtain a primitive polynomial $\tilde{P}(t)$ (in the ring-theoretic sense, i.e., a polynomial whose coefficients have greatest common divisor $1$). It is easy to see that $g \leq 2$. Now note that $\tau$ is also a zero of
$$Q(t) = mt^k-mt^{k-1}-1.$$
If $\tau$ was rational, it would have to be of the form $\pm \frac{1}{r}$ (since the denominator has to divide the leading coefficient, while the numerator has to divide the constant coefficient), contradicting the fact that $\tau(p) > 1$. Hence $\tilde{P}$ is the minimal polynomial of $\tau$, and $Q(t)$ must be divisible by $\tilde{P}(t)$ in $\IZ[t]$ (by Gauss' lemma), which implies that the constant coefficient of $\tilde{P}$ must be $\pm 1$. But this is only possible if $b(k-1) = g = 2$, i.e., $k = 3$, $b = 1$, and $a$ must be even. But then
$$C(p) = q-p + \frac{a}{b} \geq q-p+2 = 1+2q > 1,$$
and we reach a contradiction.
\end{proof}

\section{A free product}\label{sec:freeproduct}

The method of the preceding sections is generally not applicable if the Cayley graph is not a tree; however, \eqref{eq:kerdim} remains true if all cycles of $T$ have length $\equiv 2 \mod 4$ (see for instance \cite[Theorem 2]{Borovicanin:2009:nullity}). Hence it is possible to apply the same techique if the free group $\IF_\f{}$ is replaced by special free products such as $\IZ_6 \ast \IZ_6$; the Cayley graph of this group has hexagons as its only cycles and therefore satisfies the aforementioned condition. Once again, one can distinguish between (rooted) animals with the property that there exists a maximum matching that does not cover the root (type A) and (rooted) animals for which this is not the case (type B) and derive recursions. In addition, one needs to take the size of the boundary of an animal into account, which is no longer uniquely determined by the size of an animal. Hence we consider the trivariate generating functions
$$A = A(u,x,y) = \sum_{\clA \text{ of type $A$}} u^{\nu(\clA)}
x^{|\clA|}y^{|d\clA|} \quad \text{and} \quad B = B(u,x,y) = \sum_{\clA \text{ of type $B$}} u^{\nu(\clA)} x^{|\clA|}y^{|d\clA|},$$
for which one obtains, after some lengthy calculations, functional equations in analogy to those in~\eqref{eq:ABequations} as well as an integral representation analogous to~\ref{eq:our_constant} for the von Neumann dimension of the kernel of the lamplighter operator $\tilde{\opT}$ on $\CG_m\wr(\IZ_6 \ast \IZ_6)$. 

In order to determine the resulting integral, one can use the Risch-Trager algorithm, as implemented for example 
in the computer algebra system
\texttt{FriCAS}, a fork of ~\cite{axiom},
Once again, we found that there exists an algebraic antiderivative, so that we obtain an algebraic von Neumann dimension for any $m \geq 3$ (the critical
percolation parameter is $p = 0.339303$ in this case, which is a zero of the polynomial $3p^5-2p^4-2p^3-2p^2-2p+1$). It is likely that it is also irrational for all $m \geq 3$, although we do not have a proof for this conjecture. Moreover, we conjecture that in fact the kernel dimension of the adjacency operator of an arbitrary free product of
cyclic groups $\IZ_{4k+2}$ is algebraic (and probably irrational), but the computations outlined above quickly become intractable by
the present method if more complicated examples are studied.

\bibliography{Atiyah}
\bibliographystyle{plain}

\end{document}